\documentclass[11pt]{article}

\usepackage{amsmath}
\usepackage{amsthm}
\usepackage{amssymb}
\usepackage{latexsym}
\usepackage{textcomp}
\usepackage[T1]{fontenc}
\usepackage[sc]{mathpazo}
\usepackage[small,it]{caption}
\usepackage{pstricks}
\usepackage{graphicx, pst-plot, pst-node, pst-text, pst-tree}
\usepackage{color}
\usepackage{RBauthblk}
\definecolor{darkblue}{rgb}{0, 0, .4}

\usepackage[bookmarks]{hyperref}
\hypersetup{
	colorlinks=true,
	linkcolor=darkblue,
	anchorcolor=darkblue,
	citecolor=darkblue,
	urlcolor=darkblue,
	pdfpagemode=UseThumbs,
	pdftitle={Modular Decomposition and the Reconstruction Conjecture},
	pdfsubject={Combinatorics},
	pdfauthor={R. Brignall, N. Georgiou, R. Waters},
	pdfkeywords={todo}
}
\theoremstyle{plain}
\newtheorem{theorem}{Theorem}[section]
\newtheorem{corollary}[theorem]{Corollary}
\newtheorem{lemma}[theorem]{Lemma}

\newtheorem*{rc}{The Reconstruction Conjecture}

\theoremstyle{definition}

\theoremstyle{remark}

\newcounter{todocounter}

\makeatletter
\newfont{\footsc}{cmcsc10 at 8truept}
\newfont{\footbf}{cmbx10 at 8truept}
\newfont{\footrm}{cmr10 at 10truept}
\renewcommand\Affilfont{\itshape\small}

\title{Modular decomposition and the Reconstruction Conjecture}

\author[1]{Robert~Brignall}
\author[2]{Nicholas~Georgiou\thanks{Supported by the Heilbronn Institute for Mathematical Research.}}
\author{Robert~J.~Waters}
\affil[1]{Department of Mathematics and Statistics, The Open University, Milton Keynes, MK7 6AA}
\affil[2]{Department of Mathematics, University of Bristol, Bristol, BS8 1TW}

\date{27th February 2012}

\begin{document}
\maketitle

\newcommand{\I}{\mathcal{I}}
\newcommand{\skel}[1]{\mathrm{Skel}(#1)}

\begin{abstract}
We prove that a large family of graphs which are decomposable with respect to the modular decomposition can be reconstructed from their collection of vertex-deleted subgraphs.
\end{abstract}

\section{Introduction}

In this short paper, we present a strong connection between the modular decomposition of graphs and the celebrated Reconstruction Conjecture (RC). For a simple, undirected graph $G$ on $n$ vertices, a \emph{card} of $G$ is any graph of the form $G-v$ for $v\in G$. The \emph{deck} of $G$, denoted $D(G)$, is the collection of all $n$ cards.

\begin{rc}[Ulam, 1960~\cite{ulam:a-collection-of:}]
Every graph on three or more vertices is uniquely determined by its deck.
\end{rc}

The starting point of our contribution is the well-known result that states graphs on at least three vertices with more than one component are reconstructible (see, e.g.~\cite{harary:on-the-reconstr:}). The proof of this result proceeds by collecting all components from all cards in the deck, and then successively removing the largest component (which must be a component of the original graph) together with a collection of components that are ``attributable'' to this largest one. Our aim here is to apply this technique to intervals arising in the modular decomposition (for definitions, see Section~\ref{sec-definitions}).

Modular decomposition dates back at least to a 1953 talk
of Fra\"{\i}ss\'e, see~\cite{fraisse:on-a-decomposit:} for the abstract.  The first article
using modular decompositions seems to be Gallai~\cite{gallai:transitiv-orien:} who applied them to the study of
transitive orientations of graphs. It has since emerged as a versatile and powerful tool, having been rediscovered under a variety of names\footnote{For example substitution decomposition, disjunctive decomposition and $X$-join.} in settings ranging from game theory to combinatorial optimisation.

The earliest connection between RC and the modular decomposition seems to be D\"orfler in 1972~\cite{dorfler:bemerkungen-zur:}, where it is shown how a specific family of decomposable graphs may be reconstructed. See also~\cite{dorfler:some-results:,dorfler:eine-klasse:,willomitzer:ein-beitrag-zur-ulam:} for similar results, all of which are generalised by the results in this note. More recently, Skums et al~\cite{skums:operator-decomp-of:} observed a connection between RC and the ``operator decomposition'': although not exactly a special case of the modular decomposition, the main result of that paper can be derived from this one. In the study of ordered sets, a connection between RC and ``lexicographic sums'' has appeared, e.g.\ in Rampon~\cite{rampon:what-is-reconstruction:}, or Chapter 9 of Schr\"oder's book~\cite{schroder:ordered-sets:}.

However, most important to us is the work of Ill\'e~\cite{ille:recognition-problem:} who proved that indecomposable graphs are \emph{recognisable}, and Basso-Gerbelli and Ill\'e~\cite{basso:la-reconstruction:} who considered the reconstruction problem for decomposable graphs with at least two non-trivial intervals. We will review the pertinent results from these two papers in the next section.

After the preliminary definitions and results, Section~\ref{sec-skel-recon} shows how to recover the ``skeleton'' and the intervals of a decomposable graph, and these are then used in Section~\ref{sec-graph-recon} to identify many cases where decomposable graphs can be reconstructed. The final section contains some concluding remarks.

\section{Preliminary definitions and results}\label{sec-definitions}

The graphs we will consider here are all simple (no loops or multiple edges) and undirected, and (to avoid trivialities) will have at least 3 vertices unless stated otherwise. For graphs $H$ and $G$, we write $H\leq G$ to mean that $H$ is an induced subgraph of $G$.

An \emph{interval} of a graph $G$ is the induced subgraph $I$ on a set of vertices for which $N(u)\setminus V(I) = N(v)\setminus V(I)$ for every $u,v\in V(I)$. Every singleton is an interval, as is all of $G$ and the empty set.  All other interval are said to be \emph{proper}, and we say that a graph is \emph{indecomposable} if it has no proper intervals, and \emph{decomposable} otherwise.\footnote{The terms prime, irreducible, primitive or (in the analogous setting for permutations) simple have also been used in the past to mean indecomposable.} See Figure~\ref{fig-small-indec} --- note that there are no indecomposable graphs on 3 vertices.

\begin{figure}
\begin{center}
\begin{tabular}{ccccccccc}
\psset{xunit=0.01in, yunit=0.01in}
\psset{linewidth=0.005in}
\begin{pspicture}(0,0)(60,40)
\Cnode*[fillstyle=solid,radius=0.03in](0,20){1}
\Cnode*[fillstyle=solid,radius=0.03in](20,20){2}
\Cnode*[fillstyle=solid,radius=0.03in](40,20){3}
\Cnode*[fillstyle=solid,radius=0.03in](60,20){4}
\ncline{-}{1}{2}
\ncline{-}{2}{3}
\ncline{-}{3}{4}
\end{pspicture}
&\rule{10pt}{0pt}&
\psset{xunit=0.01in, yunit=0.01in}
\psset{linewidth=0.005in}
\begin{pspicture}(0,0)(80,40)
\Cnode*[fillstyle=solid,radius=0.03in](0,20){1}
\Cnode*[fillstyle=solid,radius=0.03in](20,20){2}
\Cnode*[fillstyle=solid,radius=0.03in](40,20){3}
\Cnode*[fillstyle=solid,radius=0.03in](60,20){4}
\Cnode*[fillstyle=solid,radius=0.03in](80,20){5}
\ncline{-}{1}{2}
\ncline{-}{2}{3}
\ncline{-}{3}{4}
\ncline{-}{4}{5}
\end{pspicture}
&\rule{10pt}{0pt}&
\psset{xunit=0.012in, yunit=0.012in}
\psset{linewidth=0.005in}
\begin{pspicture}(0,0)(40,40)
\Cnode*[fillstyle=solid,radius=0.03in](10,5){1}
\Cnode*[fillstyle=solid,radius=0.03in](30,5){2}
\Cnode*[fillstyle=solid,radius=0.03in](37,24){3}
\Cnode*[fillstyle=solid,radius=0.03in](20,38){4}
\Cnode*[fillstyle=solid,radius=0.03in](3,24){5}
\ncline{-}{1}{2}
\ncline{-}{2}{3}
\ncline{-}{3}{4}
\ncline{-}{4}{5}
\ncline{-}{5}{1}
\end{pspicture}
&\rule{10pt}{0pt}&
\psset{xunit=0.012in, yunit=0.012in}
\psset{linewidth=0.005in}
\begin{pspicture}(0,0)(40,40)
\Cnode*[fillstyle=solid,radius=0.03in](10,5){1}
\Cnode*[fillstyle=solid,radius=0.03in](30,5){2}
\Cnode*[fillstyle=solid,radius=0.03in](37,24){3}
\Cnode*[fillstyle=solid,radius=0.03in](20,38){4}
\Cnode*[fillstyle=solid,radius=0.03in](3,24){5}
\ncline{-}{1}{2}
\ncline{-}{2}{3}
\ncline{-}{3}{4}
\ncline{-}{4}{5}
\ncline{-}{5}{1}
\ncline{-}{3}{5}
\end{pspicture}
&\rule{10pt}{0pt}&
\psset{xunit=0.01in, yunit=0.01in}
\psset{linewidth=0.005in}
\begin{pspicture}(0,0)(60,40)
\Cnode*[fillstyle=solid,radius=0.03in](0,20){1}
\Cnode*[fillstyle=solid,radius=0.03in](20,20){2}
\Cnode*[fillstyle=solid,radius=0.03in](40,20){3}
\Cnode*[fillstyle=solid,radius=0.03in](60,20){4}
\Cnode*[fillstyle=solid,radius=0.03in](30,37){5}
\ncline{-}{1}{2}
\ncline{-}{2}{3}
\ncline{-}{3}{4}
\ncline{-}{2}{5}
\ncline{-}{3}{5}
\end{pspicture}
\end{tabular}
\end{center}
\caption{The indecomposable graphs on $4$ and $5$ vertices.}\label{fig-small-indec}
\end{figure}
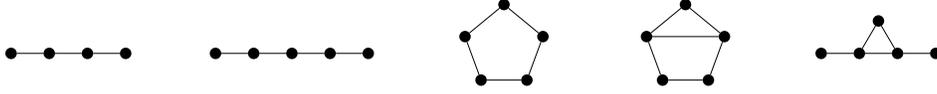

The significance of indecomposable graphs arises because they form the ``building blocks'' from which all other graphs are constructed, by means of the modular decomposition. We say that $G$ is an \emph{inflation} of a graph $K$ by the graphs $\{I_k:k\in K\}$ if $G$ is obtained by replacing each vertex $k\in K$ with the graph $I_k$ so that they form an interval in $G$. We write this as $G=K[I_k:k\in K]$.

\begin{theorem}[Modular Decomposition]\label{thm-modular-decomp}For every graph $G$, there exists a unique indecomposable graph $K$ such that $G=K[I_k:k\in K]$. Moreover, when $|K| > 2$, the graphs $I_k$ are uniquely determined.
\end{theorem}

We will refer to the unique indecomposable graph as the \emph{skeleton} of $G$, denoted $\skel{G}$. Note that the case when $|\skel{G}|=2$ corresponds to graphs with more than one component, or whose complement has more than one component: we call graphs of this form \emph{degenerate}. Since it is already known that degenerate graphs can be reconstructed, we will assume unless stated otherwise that the skeleton of any graph has size at least 4. Consequently, all the decompositions that we will consider are unique, and this enables us to exploit the modular decomposition with relative ease.

Indecomposable graphs have received considerable attention in their own right, as they have some remarkable structural properties. The first evidence of this is due to Schmerl and Trotter:

\begin{theorem}[Schmerl and Trotter~\cite{schmerl:critically-inde:}]\label{thm-schmerl-trotter} Every indecomposable graph on $n\geq 2$ vertices contains an indecomposable induced subgraph with $n-1$ or $n-2$ vertices.\end{theorem}

Moreover, Schmerl and Trotter showed that there was (up to complements) only one family of indecomposable graphs which do not have an indecomposable subgraph on $n-1$ vertices: these are called \emph{critically indecomposable}, and are illustrated in Figure~\ref{fig-crit-indecs}. A central step in Schmerl and Trotter's proof is to apply the following lemma repeatedly:

\begin{figure}
\begin{center}
\begin{tabular}{ccccccccc}
\psset{xunit=0.01in, yunit=0.01in}
\psset{linewidth=0.005in}
\begin{pspicture}(0,0)(40,80)
\Cnode*[fillstyle=solid,radius=0.03in](0,30){A1}
\Cnode*[fillstyle=solid,radius=0.03in](0,50){A2}
\Cnode*[fillstyle=solid,radius=0.03in](40,30){B1}
\Cnode*[fillstyle=solid,radius=0.03in](40,50){B2}
\ncline{-}{A1}{B1}
\ncline{-}{A2}{B1}
\ncline{-}{A2}{B2}
\end{pspicture}
&\rule{15pt}{0pt}&
\psset{xunit=0.01in, yunit=0.01in}
\psset{linewidth=0.005in}
\begin{pspicture}(0,0)(40,80)
\Cnode*[fillstyle=solid,radius=0.03in](0,20){A1}
\Cnode*[fillstyle=solid,radius=0.03in](0,40){A2}
\Cnode*[fillstyle=solid,radius=0.03in](0,60){A3}
\Cnode*[fillstyle=solid,radius=0.03in](40,20){B1}
\Cnode*[fillstyle=solid,radius=0.03in](40,40){B2}
\Cnode*[fillstyle=solid,radius=0.03in](40,60){B3}
\ncline{-}{A1}{B1}
\ncline{-}{A2}{B1}
\ncline{-}{A2}{B2}
\ncline{-}{A3}{B1}
\ncline{-}{A3}{B2}
\ncline{-}{A3}{B3}
\end{pspicture}
&\rule{15pt}{0pt}&
\psset{xunit=0.01in, yunit=0.01in}
\psset{linewidth=0.005in}
\begin{pspicture}(0,0)(40,80)
\Cnode*[fillstyle=solid,radius=0.03in](0,10){A1}
\Cnode*[fillstyle=solid,radius=0.03in](0,30){A2}
\Cnode*[fillstyle=solid,radius=0.03in](0,50){A3}
\Cnode*[fillstyle=solid,radius=0.03in](0,70){A4}
\Cnode*[fillstyle=solid,radius=0.03in](40,10){B1}
\Cnode*[fillstyle=solid,radius=0.03in](40,30){B2}
\Cnode*[fillstyle=solid,radius=0.03in](40,50){B3}
\Cnode*[fillstyle=solid,radius=0.03in](40,70){B4}
\ncline{-}{A1}{B1}
\ncline{-}{A2}{B1}
\ncline{-}{A2}{B2}
\ncline{-}{A3}{B1}
\ncline{-}{A3}{B2}
\ncline{-}{A3}{B3}
\ncline{-}{A4}{B1}
\ncline{-}{A4}{B2}
\ncline{-}{A4}{B3}
\ncline{-}{A4}{B4}
\end{pspicture}
&\rule{15pt}{0pt}&
\psset{xunit=0.01in, yunit=0.01in}
\psset{linewidth=0.005in}
\begin{pspicture}(0,0)(40,80)
\Cnode*[fillstyle=solid,radius=0.03in](0,0){A1}
\Cnode*[fillstyle=solid,radius=0.03in](0,20){A2}
\Cnode*[fillstyle=solid,radius=0.03in](0,40){A3}
\Cnode*[fillstyle=solid,radius=0.03in](0,60){A4}
\Cnode*[fillstyle=solid,radius=0.03in](0,80){A5}
\Cnode*[fillstyle=solid,radius=0.03in](40,0){B1}
\Cnode*[fillstyle=solid,radius=0.03in](40,20){B2}
\Cnode*[fillstyle=solid,radius=0.03in](40,40){B3}
\Cnode*[fillstyle=solid,radius=0.03in](40,60){B4}
\Cnode*[fillstyle=solid,radius=0.03in](40,80){B5}
\ncline{-}{A1}{B1}
\ncline{-}{A2}{B1}
\ncline{-}{A2}{B2}
\ncline{-}{A3}{B1}
\ncline{-}{A3}{B2}
\ncline{-}{A3}{B3}
\ncline{-}{A4}{B1}
\ncline{-}{A4}{B2}
\ncline{-}{A4}{B3}
\ncline{-}{A4}{B4}
\ncline{-}{A5}{B1}
\ncline{-}{A5}{B2}
\ncline{-}{A5}{B3}
\ncline{-}{A5}{B4}
\ncline{-}{A5}{B5}
\end{pspicture}
&\rule{15pt}{0pt}&
\psset{xunit=0.01in, yunit=0.01in}
\psset{linewidth=0.005in}
\begin{pspicture}(0,0)(10,80)
\rput[c](5,40){\dots}
\end{pspicture}
\end{tabular}
\end{center}
\caption{Up to complements, the single infinite family of critically indecomposable graphs.}\label{fig-crit-indecs}
\end{figure}
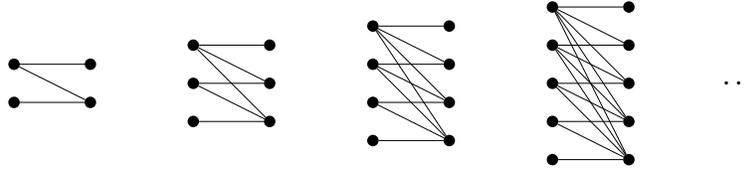

\begin{lemma}\label{lem-bootstrap}Let $G$ be an indecomposable graph on $n\geq 5$ vertices, and let $X$ be a set of vertices with $3\leq |X|\leq n-2$ and such that $G[X]$ is indecomposable. Then there are distinct vertices $u,v\in V(G)\setminus X$ such that $G[X\cup\{u,v\}]$ is indecomposable.\end{lemma}

As an example of its use, we can combine it with the following lemma to obtain a mild strengthening of Theorem~\ref{thm-schmerl-trotter}:

\begin{lemma}[Ill\'e~\cite{ille:indecomposable-:}]\label{lem-specific-vertex} Let $G$ be an indecomposable graph on $n\geq 6$ vertices, and let $v$ be any vertex of $G$. Then there exists a set $X\subseteq V(G)$ such that $v\in X$ and $3\leq |X| < n$, and $G[X]$ is indecomposable.\end{lemma}

\begin{corollary}\label{cor-indecomposable-subgraph}Let $G$ be an indecomposable graph with $n\geq 6$ vertices, and let $v$ be any vertex of $G$. Then there exists a set $X$ on $n-1$ or $n-2$ vertices, such that $v\in X$ and $G[X]$ is indecomposable.\end{corollary}

As mentioned in the introduction, it has been shown that indecomposable graphs are recognisable:

\begin{theorem}[Ill\'e~\cite{ille:recognition-problem:}]
Let $G$ and $H$ be graphs for which $D(G)=D(H)$. Then $G$ is indecomposable if and only if $H$ is indecomposable.\end{theorem}

Consequently, in this note we will assume that the graphs we reconstruct have already been recognised as decomposable.

Finally, we mention one result due to Basso-Gerbelli and Ill\'e~\cite{basso:la-reconstruction:}. It is closely related to ours, although we do not need it in the sequel. They state that if $G$ and $H$ are decomposable graphs each with at least two non-trivial intervals and for which $D(G)=D(H)$, then $\skel{G}\cong\skel{H}$ and there is a $1-1$ correspondence between the maximal non-trivial intervals of $G$ and $H$.

\section{Recovering the skeleton and intervals}\label{sec-skel-recon}

Recovering the skeleton is a relatively straightforward procedure, relying on the following lemma.

\begin{lemma}\label{lem-deck-skeletons}Let $G$ be any graph. Then for every $G-v\in D(G)$ we have $\skel{G-v}\leq\skel{G}$.\end{lemma}

\begin{proof}
Let $G=K[I_k:k\in K]$ be the modular decomposition of $G$, and consider any card $G-v$, where $v$ lies in $I_{k^*}$ for some $k^*\in K$. If $|I_{k^*}|>1$ then the result follows immediately since $G-v=K[I'_k:k\in K]$ where $I'_k=I_k$ for $k\neq k^*$ and $I'_{k^*}=I_{k^*}-v$.

Thus we may assume that $I_{k^*}$ contains only the vertex $v$. However, we may then write $G-v= (K-k^*)[I_k:k\in K-k^*]$, an inflation of $K-k^*$. If $K-k^*$ is indecomposable, then $\skel{G-v}=K-k^*$, otherwise we have $\skel{G-v}=\skel{K-k^*}\leq K$.\end{proof}

\begin{theorem}\label{thm-skel-recon}Let $G$ be a decomposable graph. Then $\skel{G}$ is reconstructible.\end{theorem}

\begin{proof}By Lemma~\ref{lem-deck-skeletons} every $G-v\in D(G)$ satisfies $\skel{G-v}\leq\skel{G}$. Moreover, since $G$ is decomposable, there exists some maximal proper interval $I$ of $G$ containing at least two vertices, and for any $v\in I$ we have $\skel{G-v}=\skel{G}$. Hence $\skel{G}$ can be obtained by taking the unique largest skeleton of all the cards in the deck.
\end{proof}

Now we know the skeleton of $G$, we need to recover the list of maximal proper intervals. Let $s(G)$ denote the number of singleton intervals in the modular decomposition of $G$, so that $s(G)=|G|$ if and only if $G$ is indecomposable. Note that $s(G)$ is reconstructible, as $s(G) = |\{G-v\in D(G): \skel{G-v}\neq\skel{G}\}|$.

The recovery of the maximal proper intervals divides into three cases:
(1) $G$ has at least two non-trivial maximal intervals; (2) $G$ has exactly one non-trivial maximal interval $I_{k^*}$ with $|I_{k^*}|\geq 3$; and (3) $G$ has exactly one non-trivial maximal interval $I_{k^*}$ with $|I_{k^*}|=2$. The next three lemmas will cover each of these cases in turn.

\begin{lemma}\label{lem-multi-interval-recon}Let $G$ be a decomposable graph for which $|\skel{G}|-s(G)\geq 2$. Then the set of maximal proper intervals of $G$ is reconstructible. Moreover, the intervals belonging to each orbit of the automorphism group of $\skel{G}$ can be identified. \end{lemma}

\begin{proof}
First set $K=\skel{G}$, define $D_K(G)=\{G-v\in D(G) : \skel{G-v}=K\}$, and let $I_1,I_2,\ldots,I_{|K|-s(G)}$ be the non-singleton maximal proper intervals of $G$, in some (arbitrary) order. Consider all the maximal proper intervals in the elements of $D_K(G)$. This consists of
\begin{itemize}
\item $s(G)(|G|-s(G))$ copies of the $s(G)$ singleton intervals in $G$,
\item $|G|-s(G)-|I_i|$ copies of each $I_i$, and
\item $D(I_i)$ for each $I_i$.
\end{itemize}
We now recover $I_1,\ldots,I_{|K|-s(G)}$ by attributability: repeatedly take any largest maximal interval graph $I_i$ from the list of maximal intervals in $D_K(G)$, and remove from the list all the elements \emph{attributable} to this interval, namely the deck $D(I_i)$.

For the second part, we note that in $D_K(G)$ one can identify to which orbit each maximal interval belongs, and thus the maximal intervals in $D_K(G)$ attributable to some maximal interval $I$ of $G$ can all be taken from the same orbit. This orbit is necessarily the orbit in which $I$ lies.
\end{proof}

The above proof does not work when $G$ has only one non-singleton maximal proper interval, as we do not see a copy of the interval in $D_K(G)$. In order to recover the interval in this case, we need to know more about the structure of $K$.

\begin{lemma}\label{lem-single-interval-3-recon} Let $G$ be a decomposable graph having exactly one maximal non-singleton interval $I_{k^*}$ of size at least $3$. Then $I_{k^*}$ is reconstructible.\end{lemma}

\begin{proof}Note first that we know the size of $I_{k^*}$, since $|I_{k^*}|=|G|-s(G)$. Moreover, we can obtain the deck $D(I_{k^*})$ since it is precisely the list of maximal non-singleton intervals in the cards of $D_K(G)$. Thus, if $I_{k^*}$ is a reconstructible graph then we are done, and in particular we may now assume that $I_{k^*}$ is neither a disconnected graph nor the complement of a disconnected graph. Thus $I_{k^*}$ is either an indecomposable graph or a non-degenerate decomposable graph.

Suppose first that $|K|\geq 6$, so by Corollary~\ref{cor-indecomposable-subgraph} there exists a proper subgraph $L$ of $K$ such that $k^*\in L$, $|L|\geq |K|-2$ and $L$ is indecomposable: fix one such $L$. In the case that $|L|=|K|-1$, there exists some card in $D(G)$ whose skeleton is precisely $L$, and whose only non-trivial interval is therefore $I_{k^*}$. Thus, we will now assume that $|L|=|K|-2$. There are two cases:
\begin{enumerate}
\item[(a)] There exists a card $H$ in $D(G)$ that is an inflation of $L$. If $H$ has two maximal proper intervals, then one will be of size $2$, and the other is $I_{k^*}$ (which we can identify as it has at least $3$ vertices). If $H$ has only one maximal proper interval, then it is either $K_2[I_{k^*},\bullet]$ or $\overline{K_2}[I_{k^*},\bullet]$ (where $\bullet$ denotes the single vertex graph), and in either case $I_k^*$ can be recovered since it is non-degenerate.

\item[(b)] There exists a degenerate card in the deck of the form $K_2[L,\bullet]$ or $\overline{K_2}[L,\bullet]$. In particular, since $k^*\in L$ we know that this card is either $K_2[L[I_{k^*},\bullet,\ldots,\bullet],\bullet]$ or $\overline{K_2}[L[I_{k^*},\bullet,\ldots,\bullet],\bullet]$. Now we know that $|L|\geq 4$ and $L$ is indecomposable, so we can recover the graph $L[I_{k^*},\bullet,\ldots,\bullet]$ from this card, and then $I_{k^*}$ is the only maximal proper interval.
\end{enumerate}

This leaves the cases $|K|=4$ and $|K|=5$, which can be verified by directly considering inflations of the graphs in Figure~\ref{fig-small-indec}. When $|K|=4$, the three cards in $D(G)\setminus D_K(G)$ all have degenerate skeletons, while $I_{k^*}$ has a non-degenerate skeleton and so can be recovered by inspection. When $|K|=5$, in all cases except the rightmost graph of Figure~\ref{fig-small-indec} there must be a card whose skeleton is $K_4$ and whose only non-singleton maximal proper interval is $I_{k^*}$. For the rightmost graph, the decomposition of all cards is degenerate while the decomposition of $I_{k^*}$ is not.
\end{proof}

\begin{lemma}\label{lem-single-interval-2-recon}Let $G$ be a decomposable graph having exactly one maximal non-singleton interval $I_{k^*}$ of size $2$. Then $I_{k^*}$ is reconstructible.\end{lemma}

\begin{proof}
First suppose that the skeleton $K$ (which we can reconstruct by Theorem~\ref{thm-skel-recon}) is not critically indecomposable, so there exists $k'\in K$ such that $K-k'$ is indecomposable. If there is a unique vertex $k'$ with this property, then we can recognise the case where $k'=k^*$: every card in $D(G)\setminus D_K(G)$ is an inflation of some indecomposable with at most $|K|-2$ vertices. In this case, since $k'=k^*$ is unique we can identify which point of $K$ needs to be inflated, and using $|E(G)|$ (which is reconstructible) we can determine whether $I_{k^*}=K_2$ or $\overline{K_2}$.\footnote{Note that we have in fact reconstructed the graph in the case where $k'=k^*$ is unique.} Thus we can assume that there is some $k'\neq k^*$ with $K-k'$ indecomposable. In $D(G)$, there is a graph whose skeleton is exactly $K-k'$, and the only non-singleton maximal proper interval of any such card is $I_{k^*}$.

The case where $K$ is critically indecomposable follows by direct case analysis of inflations of the family of graphs in Figure~\ref{fig-crit-indecs}: note that each graph in this family has exactly two points whose removal leaves a graph which is the disjoint union of an isolated vertex and a smaller graph from the same family, and an analogous argument applies to the complement of this family. As at least one these points cannot be $k^*$, there is at least one card (and at most two) in $D(G)$ with an isolated vertex, and the other component of the graph has exactly one non-singleton maximal proper interval, namely $I_{k^*}$.
\end{proof}

\section{Reconstructing decomposable graphs}\label{sec-graph-recon}

Building on the results of the previous section, we can now state and prove the primary observation of this section. We use $\mathrm{Orb}_K(k^*)$ to denote the orbit of a vertex $k^*\in K$ under the action of the automorphism group of $K$.

\begin{theorem}\label{thm-non-hereditary-recon}Let $G=K[I_k:k\in K]$ be a decomposable graph with $|K|\leq |G|-2$. If there exists a non-singleton interval $I_{k^*}$ and $u\in I_{k^*}$ for which $I_{k^*}-u\not\in\{I_j : j \in \mathrm{Orb}_K(k^*)\}$, then $G$ is reconstructible.\end{theorem}

\begin{proof}
First, note that the condition $|K|\leq |G|-2$ ensures that $G$ either has at least two non-singleton maximal proper intervals, or a single maximal proper interval with at least three vertices.

Suppose first that $G$ has at least two non-singleton maximal proper intervals. By Lemma~\ref{lem-multi-interval-recon} we can recover all of $\{I_k:k\in K\}$ and identify which intervals belong to each orbit of $K$. We now reconstruct $G$ by taking any card of $D_K(G)$ for which there is a point of $\mathrm{Orb}_K(k^*)$ which has been inflated by $I_{k^*}-u$
(note that there may be more than one in the case where $I_{k^*}-u\cong I_{k^*}-v$ for vertices $u\neq v$), and replacing the maximal proper interval $I_{k^*}-u$ of the card with $I_{k^*}$.

In the case where there is exactly one non-singleton interval with at least three vertices, we use Lemma~\ref{lem-single-interval-3-recon} to recover $I_{k^*}$. It is now trivial to reconstruct $G$: we find any graph in $D_K(G)$ and replace the only non-singleton maximal proper interval with $I_{k^*}$.
\end{proof}

As an example of its use, we identify two corollaries in contrasting settings:

\begin{corollary}\label{cor-trivial-aut-recon}Let $G$ be a decomposable graph whose skeleton $K$ has trivial automorphism group and satisfies $|K|\leq |G|-2$. Then $G$ is reconstructible.\end{corollary}

\begin{corollary}\label{cor-vertex-tran-recon}Any graph whose skeleton is vertex-transitive is reconstructible.\end{corollary}

\begin{proof}[Proof of Corollary~\ref{cor-vertex-tran-recon}]
Consider a graph $G$ whose skeleton $K$ is vertex-transitive. If $|K|=|G|$ then $G$ is regular, and hence can be reconstructed (see~\cite{kelly:a-congruence-th:}). If $|K|=|G|-1$, then $G$ is an inflation of $K$ by a single interval of size two. Having recovered the interval using Lemma~\ref{lem-single-interval-2-recon}, we inflate any vertex of $K$ by this interval to recover $G$.

Thus we can assume $|K|\leq |G|-2$. The only case that is not covered by Theorem~\ref{thm-non-hereditary-recon} is where the set $\I = \{I_k:k\in K\}$ is hereditary.\footnote{We say that $\I$ is \emph{hereditary} if for any $I\in \I$, whenever $J$ is an induced subgraph of $I$ then $J\in \I$.} In particular, this means that there is at least one singleton interval in $\I$, so there is some card $H\in D(G)\setminus D_K(G)$ arising from deleting one of these singletons. However, since $K$ is vertex transitive it cannot be critically indecomposable, and it then follows by transitivity that $K-k$ is indecomposable for every $k\in K$. Hence the skeleton of $H$ is $K-k$ (for any $k$), and
the maximal proper intervals of $H$ consist of all the maximal proper intervals of $G$, except that there is one singleton missing.  We can now reconstruct $G$ from $H$ since we can reconstruct the regular graph $K$ from $K-k$.
\end{proof}

\paragraph{Single proper interval of size two.} Suppose $G$ is a decomposable graph with skeleton $K$ satisfying $|K|=|G|-1$. Lemma~\ref{lem-single-interval-2-recon} does not in general identify which orbit of $K$ the single inflated vertex arises from. However, by imposing some further conditions on $K$ we can complete this step in the reconstruction.

Bollob\'as~\cite{bollobas:almost-every-graph:} defines a family of graphs $\mathcal{F}$ with the following property: $G \in \mathcal{F}$ if $G$ has trivial automorphism group, and all subgraphs of size $|G|-1$ and $|G|-2$ embed uniquely in~$G$.  Bollob\'as shows that this family contains almost all graphs, and in fact only three cards are needed to reconstruct any graph from the family.

Requiring that $K$ lies in $\mathcal{F}$ is a sufficient condition to identify which vertex of $K$ to inflate, but we can take a slightly more general class of graphs, which explicitly highlights the properties that we require. Define the family $\mathcal{G}$ to consist of all graphs $G$ satisfying the following two conditions:
\begin{enumerate}
\item[(1)] No \emph{pseudo-similar} vertices: For $u,v\in G$, if $G-u\cong G-v$, then $u\in\mathrm{Orb}_G(v)$
\item[(2)] For $u,v,w\in G$, if $u\in\mathrm{Orb}_{G-w}(v)$, then $u\in\mathrm{Orb}_G(v)$.
\end{enumerate}
Note that $\mathcal{F}\subset\mathcal{G}$: If $G\in\mathcal{F}$, then $G$ and all cards in $D(G)$ must have trivial automorphism groups, from which conditions (1) and (2) follow. On the other hand, $\mathcal{G}$ contains graphs which are not in $\mathcal{F}$: for example, $\mathcal{G}$ contains all vertex transitive graphs.

\begin{theorem}\label{thm-size-2-recon} Let $G$ be a decomposable graph whose skeleton $K$ satisfies $|K|=|G|-1$, and $K\in\mathcal{G}$. Then $G$ is reconstructible.
\end{theorem}

\begin{proof}
By Lemma~\ref{lem-single-interval-2-recon}, we only need to identify the vertex $k^*$ of $K$ which needs to be inflated by the interval of size two, and in fact it suffices to identify any vertex in $\mathrm{Orb}_K(k^*)$. Observe that $\mathcal{G}$ does not contain the critically indecomposable graphs (although it can be shown by brute force that inflations of these can be reconstructed), and therefore $K$ is not critically indecomposable. Hence, there exists $k'\in K$ such that $K-k'$ is indecomposable.

As in the proof of Lemma~\ref{lem-single-interval-2-recon}, if $k'$ is unique with this property and $k'=k^*$, then we can uniquely identify $k^*$ in $K$. Thus we may assume that we have $k'\neq k^*$ with $K-k'$ indecomposable. In the deck of $G$, this means that there is a card that is an inflation of $K-k'$ by an interval of size two: the vertex of $K-k'$ that is inflated to form this card is $k^*$.

Fix any embedding $\phi:K-k'\hookrightarrow K$. We will be done if we can show that $\phi(k^*)\in \mathrm{Orb}_K(k^*)$, as inflating $\phi(k^*)$ is then isomorphic to inflating the vertex $k^*$. Let $k''\in K$ represent the unique vertex of $K$ that is not in the image of $\phi$, so that $\phi$ is an isomorphism from $K-k'$ to $K-k''$. By condition~(1), we have $k''\in\mathrm{Orb}_K(k')$, so there exists an automorphism $\psi$ of $K$ for which $\psi(k'')=k'$.

Now, $\psi$ is an isomorphism from $K-k''$ to $K-k'$, so $\psi\circ\phi$ is an automorphism of $K-k'$. Thus $\psi\circ\phi(k^*)\in\mathrm{Orb}_{K-k'}(k^*)$, and so by condition (2) $\psi\circ\phi(k^*)\in\mathrm{Orb}_K(k^*)$. However, $\psi$ was an automorphism of $K$, and so we conclude that $\phi(k^*)\in\mathrm{Orb}_K(k^*)$ as required.
\end{proof}

Combining this theorem with Corollary~\ref{cor-trivial-aut-recon} and the results of Bollob\'as~\cite{bollobas:almost-every-graph:}, we have that \emph{any} graph $G$ whose skeleton $K$ lies in the family of graphs $\mathcal{F}$ is reconstructible. Note that this includes graphs which are not themselves members of $\mathcal{F}$, for example by inflating a single vertex of $K\in\mathcal{F}$ by any graph which does not lie in $\mathcal{F}$.

\section{Concluding remarks}

\paragraph{A possible route to proving RC?} As exemplified by results such as those given in Section~\ref{sec-definitions}, indecomposable graphs have an emerging structure theory that very naturally fits with graph reconstruction. We offer, therefore, some faint optimism that this theory could be developed and exploited to reconstruct indecomposable graphs. For decomposable graphs, we have reduced the problem of graph reconstruction to considering two specific families. In both cases, the properties exploited to prove reconstruction in indecomposable graphs would invariably help.

\paragraph{Single interval of size two.} In the case $|K|=|G|-1$, Theorem~\ref{thm-size-2-recon} applies to graphs whose skeletons come from the family $\mathcal{G}$. If we concentrate on properties of indecomposable graphs, we can relax the conditions on the skeleton $K$ to the following: There exists a vertex $k'\in K$, such that $K-k'$ is indecomposable, $K-k \cong K-k'$ implies $k \in \mathrm{Orb}_K(k')$, and $u \in \mathrm{Orb}_{K-k'}(v)$ implies $u \in \mathrm{Orb}_K(v)$ for all $u,v \in K-k'$. What can be said about the structure of indecomposable graphs that do not have this property?

\paragraph{Hereditary orbits.} In the case where $|K|\leq |G|-2$, Theorem~\ref{thm-non-hereditary-recon} cannot be applied precisely when the set of maximal intervals belonging to each orbit is hereditary. A useful observation, which enabled us to deduce Corollary~\ref{cor-vertex-tran-recon}, is that every one of these sets necessarily contains a singleton interval, and so among the graphs in $D(G)$ we can find one which is an inflation of $K-k$ for every $k\in K$. However, we cannot guarantee that each card can be uniquely expressed as an inflation in this way, since not every $K-k$ will be indecomposable.

%
%
%
%
%
%

\bibliographystyle{acm}
\bibliography{../refs}

\end{document}